\documentclass[twoside]{article}
\usepackage[english]{babel}
\usepackage[latin1]{inputenc}
\usepackage{amsmath,amssymb}
\usepackage{amsfonts}
\usepackage{latexsym}
\usepackage{graphics, graphicx}
\usepackage{upref}
\usepackage{psfrag,rotating}
\usepackage{t1enc}
\usepackage{cite}
\usepackage[]{color}
\usepackage[T1]{fontenc} 
\usepackage{lmodern}  
\usepackage{comment}     
\usepackage{amsthm, epsfig,tabularx}

\allowdisplaybreaks[1]

\numberwithin{equation}{section}

\newcommand{\B}{\ensuremath{B^s_{p,q}}}

\newcommand{\Ad}{\ensuremath{\mathbf{A}^s_{p,q}}}
\newcommand{\Bd}{\ensuremath{\mathbf{B}^s_{p,q}}}
\newcommand{\Bdx}[3]{\ensuremath{\mathbf{B}^{#1}_{#2,#3}}}

\definecolor{gray}{rgb}{0.19,0.19,0.19}

\newcommand{\Fd}{\ensuremath{\mathbf{F}^s_{p,q}}}

\newcommand{\F}{\ensuremath{F^s_{p,q}}}
\newcommand{\ud}{\mathrm{d}}
\newcommand{\uD}{\mathrm{D}}

\newcommand{\real}{\ensuremath{\mathbb{R}}}
\newcommand{\rn}{\ensuremath{\real^{n}}}
\newcommand{\nat}{\ensuremath{\mathbb{N}}}
\newcommand{\no}{\ensuremath{\nat_0}}
\newcommand{\non}{\ensuremath\mathbb{N}^n_0}
\newcommand{\zn}{\ensuremath{\mathbb{Z}^n}}
\newcommand{\comp}{\ensuremath{\mathbb{C}}}
\newcommand{\dint}{\mathrm{d}}

\newcommand{\SpRn}{\ensuremath{\mathcal{S}'(\rn)}}
\newcommand{\bit}{\begin{itemize}}
\newcommand{\eit}{\end{itemize}}
\newcommand{\beq}{\begin{equation}}
\newcommand{\eeq}{\end{equation}}

\newcommand{\supp}{\mathrm{supp}\,}

\newcommand{\id}{\mathrm{id}}

\theoremstyle{remark}

\newcommand\R{\mathbb{R}}

\newlength{\fixboxwidth}
\newtheorem{lemma}{Lemma}[section]
\newtheorem{proposition}[lemma]{Proposition}
\newtheorem{theorem}[lemma]{Theorem}
\newtheorem{corollary}[lemma]{Corollary}
\newtheorem{definition}[lemma]{Definition}
\newtheorem{remark}[lemma]{Remark}
\newtheorem{example}[lemma]{Example}

\newcommand{\proofstart}{\mbox{P\,r\,o\,o\,f\, :\quad}}
\newcommand{\proofend}{\nopagebreak\hfill\raisebox{0.3em}{\fbox{}}\\}

\newcommand{\Res}{\mathrm{Re}\,}

\newcommand{\ex}{\mathrm{Ex}\,}

\begin{document}

\title{Homogeneity property of Besov and Triebel-Lizorkin spaces}
\author{Cornelia Schneider and Jan Vyb\'iral}

\date{\today}

\maketitle 

\begin{abstract}
We consider the classical Besov and Triebel-Lizorkin spaces defined via differences
and prove a homogeneity property for functions with bounded support in the frame of these spaces.
As the proof is based on compact embeddings between the studied function spaces 
we present also some results on the entropy numbers of these embeddings.
Moreover, we derive some applications in terms of pointwise multipliers.
\end{abstract}

\vskip.5cm

{\bf 2010 AMS Classification:} 46E35, 47B06.


\footnotetext{The second author acknowledges the financial support provided by the START-award ``Sparse Approximation and
Optimization in High Dimensions'' of the Fonds zur F\"orderung der wissenschaftlichen Forschung (FWF, Austrian Science Foundation).}

\section*{Introduction}

The present note deals with classical Besov spaces $\Bd(\rn)$ and Triebel-Lizorkin spaces $\Fd(\rn)$ defined via differences, 
briefly denoted as B- and F-spaces in the sequel. We study the properties of the dilation operator, which is defined
for every $\lambda>0$ as
$$
T_\lambda:f\to f(\lambda\cdot).
$$
The norms of these operators on Besov and Triebel-Lizorkin spaces were studied already in \cite{Bo83} and
\cite[Sections 2.3.1 and 2.3.2]{ET} with complements given in \cite{Vyb08}, \cite{sch08b}, and \cite{SV09}.

We prove the so-called {\em homogeneity property}, showing that for $s>0$ and $0<p,q\leq\infty$,
\beq\label{eq:1}
\|f(\lambda\cdot)|\Bd(\rn)\|\sim \lambda^{s-\frac np}\|f|\Bd(\rn)\|,
\eeq
for all $0<\lambda\leq 1$ and all 
\[
f\in\Bd(\rn)\quad \text{with}\quad \supp f\subset \{x\in\rn:|x|\leq\lambda\}.
\]
The same property holds true for the spaces $\Fd(\rn)$.  This extends and completes \cite{CLT07}, where corresponding results 
for the spaces $\B(\rn)$, defined via Fourier-analytic tools, were established, which coincide with our spaces $\Bd(\rn)$ 
if \ $s>\max\left(0,n\left(\frac 1p-1\right)\right)$\ .   Concerning the corresponding F-spaces $\F(\rn)$, the same homogeneity 
property had already been established in \cite[Cor.~5.16, p.~66]{T-func}. \\
Our results yield immediate applications in terms of pointwise multipliers. 
Furthermore, we remark that the homogeneity property is closely related with questions concerning refined localization,
non-smooth atoms, local polynomial approximation, and scaling properties. This is out of our scope for the time being. But  we use  this property in the forthcoming paper \cite{SV11}  in connection with non-smooth atomic decompositions in function spaces. 

Our proof of \eqref{eq:1} is based on compactness of embeddings between the function spaces under investigation. Therefore we use
this opportunity to present some closely related results on entropy numbers of such embeddings.

This note is organized as follows. We start with the necessary definitions and the results about entropy numbers in 
Section \ref{S-1}. Then we focus on equivalent quasi-norms for the elements of certain subspaces of $\Bd(\rn)$ and $\Fd(\rn)$, 
respectively, from which the homogeneity property will follow almost immediately in Section \ref{S-2}.  The last section states 
some applications in terms of pointwise multipliers.  

\section{Preliminaries}
\label{S-1}
We use standard notation. Let $\nat$ be the collection of all natural numbers
and let $\no = \nat \cup \{0 \}$. Let $\rn$ 
be Euclidean $n$-space, $n \in \nat$, $\comp$ the complex plane. The set of
multi-indices $\beta=(\beta_1, \dots, \beta_n)$, $\beta_i\in\no$, $i=1, \dots,
n$, is denoted by $\non$, with $|\beta|=\beta_1 + \cdots + \beta_n$, as usual.
We use the symbol \ '$\lesssim$'\  in 
\[
a_k \lesssim b_k \quad \mbox{or} \quad \varphi(x) \lesssim \psi(x)
\]
always to mean that there is a positive number $c_1$ 
such that
\[
a_k \leq c_1\,b_k  \quad \mbox{or} \quad
\varphi(x) \leq c_1\,\psi(x) 
\]
for all admitted values of the discrete variable $k$ or the
continuous variable $x$, where $(a_k)_k$, $(b_k)_k$ are
non-negative sequences and $\varphi$, $\psi$ are non-negative
functions. 
We use the equivalence \ `$\sim$' \ in
\[a_k \sim b_k \quad \mbox{or} \quad \varphi(x) \sim \psi(x)\]
for 
\[
a_k\lesssim b_k\quad \text{and} \quad b_k\lesssim a_k \qquad \text{or}\qquad 
\varphi(x) \lesssim \psi(x)\quad \text{and} \quad \psi(x) \lesssim \varphi(x).
\]
If $a\in\real$, then $a_+ := \max(a,0)$ and $[a]$
denotes the integer part of $a$.\\
Given two (quasi-) Banach spaces $X$ and $Y$, we write $X\hookrightarrow Y$
if $X\subset Y$ and the natural embedding of $X$ in $Y$ is continuous. 
All unimportant positive constants will be denoted by $c$, occasionally with
subscripts. For convenience, let both $ \dint x $ and $ |\cdot| $ stand for the
($n$-dimensional) Lebesgue measure in the sequel. 
$L_p(\rn)$, with $0<p\leq\infty$, stands for the usual quasi-Banach space with respect to the Lebesgue measure, quasi-normed by
\[
\|f|L_p(\rn)\|:=\left(\int_{\rn}|f(x)|^p\ud x\right)^{\frac 1p}
\]
with the appropriate modification if $p=\infty$. Moreover, let $\Omega$ denote a domain in $\rn$. Then  $L_p(\Omega)$ is the collection of all complex-valued Lebesgue measurable functions in $\Omega$ such that 
 \[
\|f|L_p(\Omega)\|:=\left(\int_{\Omega}|f(x)|^p\ud x\right)^{\frac 1p}
\]
(with the usual modification if $p=\infty$) is finite. \\
Furthermore, $B_R$ stands for an open ball with radius $R>0$ around the origin, 
\beq\label{op_ba}
B_R=\{
x\in\rn:\, |x|<R
\}.
\eeq

Let $Q_{j,m}$ with $j\in\no$ and $m\in\zn$ denote a cube in
$\rn$ with sides parallel to the axes of coordinates, centered at
$2^{-j}m$, and with side length $2^{-j+1}$. For a cube $Q$ in $\rn$ and $r>0$,
we denote by  $rQ$  the cube in $\rn$ concentric with $Q$ and with side length
$r$ times the side length of $Q$. 
Furthermore, $\chi_{j,m}$ stands for the characteristic function of $Q_{j,m}$.\\

\subsubsection*{Function spaces defined via differences}

If $f$
is an arbitrary function on $\rn$, $h\in\rn$ and $r\in\nat$, then
\[
(\Delta_h^1 f)(x)=f(x+h)-f(x) \quad \text{and} \quad
(\Delta_h^{r+1} f)(x)=\Delta_h^1(\Delta_h^r f)(x)
\]
are the usual iterated differences\index{difference operator!$\Delta^r_hf$}. 
Given a function $f\in L_p(\rn)$ the \textit{$r$-th modulus of smoothness}\index{modulus of smoothness!$\omega_r(f,t)_p$} is defined
by
\begin{equation}
\omega_r(f,t)_p=\sup_{|h|\leq t} \|\Delta_h^r f\mid L_p(\rn)\|, \quad
t>0, \quad 0< p\leq \infty, \label{modulus}
\end{equation}
and 
\begin{equation}
d^r_{t,p}f(x)=\left(t^{-n}\int_{|h|\leq t}|(\Delta^r_hf)(x)|^p\ud h\right)^{1/p}, \quad
 t>0, \quad 0<p<\infty, \label{ball_means}
\end{equation}
denotes its \textit{ball means}\index{ball means $d^r_{t,p}f$}.

\begin{definition}\label{defi-Bd}\label{defi-Fd}
\begin{enumerate}
\item[{\hfill\upshape\bfseries (i)\hfill}] 
Let $0<p,q\leq \infty$, $s>0$, and $r\in\nat$ such that $r>s$. Then the Besov
space $\Bd(\rn)$\index{Besov spaces!of type $\Bd(\rn)$} contains all $f\in L_p(\rn)$ such that
\beq
\|f|\Bd(\rn)\|_r = \|f|L_p(\rn)\| + \left(\int_0^1 t^{-sq} \omega_r(f,t)_p^q
\ \frac{\dint t}{t}\right)^{1/q}
\label{B-diff}
\eeq
$($with the usual modification if $q=\infty)$ is finite.
\item[{\hfill\upshape\bfseries (ii)\hfill}] 
Let $0<p< \infty$, $0<q\leq\infty$, $s>0$, and $r\in\nat$ such that $r>s$. Then $\Fd(\rn)$\index{Triebel-Lizorkin spaces!of type $\Fd(\rn)$} is the collection of all $f\in L_p(\rn)$ such that
\beq
\|f|\Fd(\rn)\|_r = \|f|L_p(\rn)\| + \left\|\left(\int_0^1 t^{-sq}{d_{t,p}^rf(\cdot)}^q\frac{\ud t}{t}\right)^{1/q} |L_p(\rn)\right\|
\label{F-diff}
\eeq
$($with the usual modification if $q=\infty)$ is finite.
\end{enumerate}
\end{definition}

\begin{remark}{
\label{rem-B-diff}\label{rem-F-diff}
These are the {\em classical Besov} and {\em Triebel-Lizorkin spaces}, in particular, when $1\leq
p,q\leq\infty$ ($p<\infty$ for the F-spaces) and $s>0$. We shall sometimes write $\Ad(\rn)$ when both scales of spaces $\Bd(\rn)$ and $\Fd(\rn)$ are concerned simultaneously. \\
Concerning the spaces $\Bd(\rn)$, the study for all admitted $s$, $p$ and $q$ goes back to \cite{St-O-3}, we
 also refer to \cite[Ch. 5, Def. 4.3]{BS} and \cite[Ch. 2, \S 10]{dVL}. There are as well
many older references in the literature devoted to the cases $p,q\geq
1$. \\
The approach by differences for the spaces $\Fd(\rn)$ has been described in detail in \cite{T-F1} for those spaces which can also be considered as subspaces of $\SpRn$. 
Otherwise one finds in \cite[Section 9.2.2, pp. 386--390]{T-F3} the necessary explanations and references to the relevant literature. \\
Definition~\ref{defi-Bd} is independent of $r$, meaning that different values
of $r>s$ result in norms which are equivalent.  This justifies our omission of $r$ in the sequel. Moreover, the integrals $\int_0^1$ can be replaced by $\int_0^{\infty}$ resulting again in equivalent quasi-norms, cf. \cite[Sect.~2]{dVP}. \\
The spaces are quasi-Banach spaces (Banach spaces if $p,q\geq 1$).
Note that we deal with subspaces of $L_p(\rn)$, in particular, for $s>0$ and $0<q\leq\infty$, we have the embeddings
\begin{displaymath}
 \Ad(\rn)\hookrightarrow L_p(\rn),
\end{displaymath}
where $0<p\leq \infty$ ($p<\infty$ for F-spaces). 
Furthermore, the B-spaces are closely linked with the Triebel-Lizorkin spaces via
\beq\label{B-F-B}
\mathbf{B}^s_{p,\min(p,q)}(\rn)\hookrightarrow \Fd(\rn)\hookrightarrow \mathbf{B}^s_{p,\max(p,q)}(\rn),
\eeq
cf. \cite[Prop.~1.19(i)]{sch08a}. 
The classical scale of Besov spaces contains many well-known function
spaces\index{Hoelder-Zygmund spaces@H\"older-Zygmund spaces ${\mathcal C}^s$}. For example, if $p=q=\infty $, one recovers the H\"older-Zygmund spaces $\ {\mathcal C}^s(\rn)$, i.e., 
\begin{equation}\label{B=Zyg} 
 \Bdx{s}{\infty}{\infty}(\rn) =  {\mathcal C}^s(\rn), \quad s>0.
\end{equation}


\label{HN07}


Recent results by {\sc Hedberg, Netrusov}\label{Bq-emb} 
  \cite{H-N} on atomic decompositions and by {\sc Triebel}
  \cite[Sect.~9.2]{T-F3} on the reproducing formula provide an equivalent characterization of Besov spaces $\Bd(\rn)$  using {\em subatomic decompositions}, which
introduces  
$\Bd(\rn)$ as those $f\in L_p(\rn)$
which can be represented as
\[
f(x) = \sum_{\beta\in\non} \sum_{j=0}^\infty \sum_{m\in\zn} \lambda_{j,m}^\beta
k^{\beta}_{j,m}(x),\quad x\in\rn,
\]
with coefficients $\lambda = \{\lambda^\beta_{j,m} \in \comp: \beta\in\non, j\in\no,
  m\in\zn\} $ belonging to some appropriate sequence space $b^{s,\varrho}_{p,q}$ 
defined as  
        \begin{equation}
                b_{p,q}^{s,\varrho}:=\big\{\lambda: \|\lambda|b^{s,\varrho}_{p,q}\|<\infty\big\}
        \end{equation}
        where
        \begin{equation}\label{defi-Bq}
         \|\lambda|b^{s,\varrho}_{p,q}\|=\sup_{\beta\in\non}2^{\varrho|\beta|} 
         \left(\sum_{j=0}^{\infty}2^{j(s-n/p)q}\left(\sum_{m\in\zn}|\lambda^{\beta}_{j,m}|^p\right)^{q/p}
          \right)^{1/q},   
        \end{equation}
         $s>0$, $0<p,q\leq\infty$ 
        $($with the usual modification if $p=\infty$ and/or $q=\infty)$,  $\varrho\geq 0$, 
  and $ k^{\beta}_{j,m}(x)$ are
  certain standardized building blocks (which are universal).  
This subatomic characterization  will turn out to be quite useful when studying entropy numbers. 


}

\end{remark}

In terms of pointwise multipliers in $\Bd(\rn)$ the following is known. 

\begin{proposition}\label{prop-mult-diffeo}
Let $0<p,q\leq\infty$, $s>0$, $k\in\nat$ with $k>s$, and  
let $h\in C^k(\rn)$. Then 
\[
f\longrightarrow hf
\] 
is a linear and bounded operator from $\Bd(\rn)$ into itself.
\end{proposition}

The proof relies on atomic decompositions of the spaces $\Bd(\rn)$,  cf. \cite[Prop.~2.5]{sch08c}. 
We will generalize this result in Section \ref{S-PoMu} as an application of our homogeneity property. 


\subsubsection*{Function spaces on domains $\Omega$}

Let $\Omega$ be a domain in $\rn$. 
We define spaces $\Ad(\Omega)$ by restriction of the corresponding spaces on $\rn$, i.e. $\Ad(\Omega)$ is the collection of all $f\in L_p(\Omega)$ such that there is a $g\in\Ad(\rn)$ with $g\big|_\Omega=f$. Furthermore,
\[
\|f|\Ad(\Omega)\|=\inf \|g|\Ad(\rn)\|,
\]
where the infimum is taken over all $g\in\Ad(\rn)$ such that the restriction $g\big|_\Omega$ to $\Omega$ coincides in $L_p(\Omega)$ with $f$. \\

In particular, the subatomic characterization for the spaces $\Bd(\rn)$ from Remark \ref{HN07} 
carries over. For further details on this subject we refer to \cite[Sect~2.1]{sch10}. \\

Embeddings results between the spaces $\Bd(\rn)$ hold also for the spaces $\Bd(\Omega)$, since they  are defined by 
restriction of the corresponding spaces on $\rn$. 
Furthermore, these results can be improved, if we assume $\Omega\subset \rn$ to be bounded.

\begin{proposition}\label{Bcal-emb}
Let $0<s_2<s_1<\infty$, $0<p_1, p_2, q_1,q_2\leq\infty$, and $\Omega\subset \rn$ be bounded. If
\beq
\delta_+ = s_1-s_2-d\left(\frac{1}{p_1}-\frac{1}{p_2}\right)_+>0,
\label{delta-5-n}
\eeq 
we have the embedding
\beq\label{ext-bd-dom-emb}
                \mathbf{B}^{s_1}_{p_1,q_1}(\Omega)\hookrightarrow \mathbf{B}^{s_2}_{p_2,q_2}(\Omega).
\eeq

\end{proposition}

\proofstart If $p_1\leq p_2$ the embedding follows from \cite[Th.~1.15]{HS08}, since the spaces on $\Omega$ are defined by restriction of their counterparts on $\rn$. Therefore it remains to show that for 
 $p_1>p_2$ we have the embedding
\beq\label{ext-emb-bd-dom}
\mathbf{B}^{s_2}_{p_1,q_2}(\Omega)\hookrightarrow \mathbf{B}^{s_2}_{p_2,q_2}(\Omega). 
\eeq
Let $\psi\in D(\rn)$ with support in the compact set $\Omega_1$ and
\[
\psi(x)=1 \quad \text{if}\quad x\in\bar{\Omega}\subset \Omega_1. 
\]
Then for $f\in \mathbf{B}^{s_2}_{p_1,q_2}(\Omega)$, there exists $g\in\mathbf{B}^{s_2}_{p_1,q_2}(\rn)$ with 
\[
g\big|_{\Omega}=f \qquad \text{and}\qquad \|f|\mathbf{B}^{s_2}_{p_1,q_2}(\Omega)\|\sim \|g|\mathbf{B}^{s_2}_{p_1,q_2}(\rn)\|.
\]
We calculate
\begin{align}
 \|f|\mathbf{B}^{s_2}_{p_2,q_2}(\Omega)\|
&\leq \|\psi g|\mathbf{B}^{s_2}_{p_2,q_2}(\rn)\|\notag\\
&\leq  \|\psi g|\mathbf{B}^{s_2}_{p_1,q_2}(\rn)\|\notag\\
&\leq c_\psi \|g|\mathbf{B}^{s_2}_{p_1,q_2}(\rn)\|\sim \|f|\mathbf{B}^{s_2}_{p_1,q_2}(\Omega)\|.\label{bd-est-1}
\end{align}
The last inequality in \eqref{bd-est-1} follows from Proposition \ref{prop-mult-diffeo}. In the 2nd step we used \eqref{B-diff}  together with the fact that 
\[
\|\Delta^r_h(\psi g)|L_{p_2}(\rn)\|\leq c_{\Omega_1}\|\Delta^r_h(\psi g)|L_{p_1}(\rn)\|, \qquad p_1>p_2,
\]
which follows from H\"older's inequality since $\supp \psi g \subset \Omega_1$ is compact.
\smallskip 
\proofend

\subsubsection*{Entropy numbers}

In order to prove the homogeneity results later on we have to rely on the compactness of embeddings between B-spaces,  $\Bd(\Omega)$, and  F-spaces,  $\Fd(\Omega)$, respectively. This will be established with the help of entropy numbers. 
We briefly introduce the concept and collect some properties afterwards.\\

Let $X$ and $Y$ be quasi-Banach spaces and $T: X\rightarrow Y$ be a bounded linear operator. If additionally, $T$ is continuous we write $T\in L(X,Y)$. Let \index{unit ball}
\ $
U_X=\{x\in X: \|x|X\|\leq 1\}
$ \ 
denote the unit ball in the quasi-Banach space $X$. An operator $T$ is called compact if for any given $\varepsilon>0$ we can cover the image of the unit ball $U_X$ with finitely many balls in $Y$ of radius $\varepsilon$.

\begin{definition}
 Let $X,Y$ be quasi-Banach spaces and let $T\in L(X,Y)$. Then for all $k\in\nat$, the $k$th dyadic entropy number $e_k(T)$ of $T$ is defined by\index{entropy numbers $e_k$}
\[
e_k(T)=\inf \left\{\varepsilon >0: \quad T(U_X)\subset \bigcup_{j=1}^{2^{k-1}}(y_j+\varepsilon U_Y)\quad \text{for some}\quad y_1,\dots, y_{2^{k-1}}\in Y\right\},
\]
where $U_X$ and $U_Y$ denote the unit balls in $X$ and $Y$, respectively.
\end{definition}

These numbers have various elementary properties which are summarized in the following lemma.

\begin{lemma}\label{lemma-entropy}\index{entropy numbers $e_k$!properties}
 Let $X,Y$ and $Z$ be quasi-Banach spaces, let $S,T\in L(X,Y)$ and $R\in L(Y,Z)$.
\bit
\item[{\hfill\bf (i)\hfill}] \textbf{$($Monotonicity$)$} $\|T\|\geq e_1(T)\geq e_2(T)\geq \dots\geq 0$. Moreover, $\|T\|=e_1(T)$, provided that $Y$ is a Banach space.
\item[{\hfill\bf (ii)\hfill}] \textbf{$($Additivity$)$} If $Y$ is a $p$-Banach space $(0<p\leq 1)$, then for all $j,k\in \nat$
\[
e^p_{j+k-1}(S+T)\leq e_j^p(S)+e_k^p(T).
\]
\item[{\hfill\bf (iii)\hfill}] \textbf{$($Multiplicativity$)$} For all $j,k\in\nat$ 
\[
e_{j+k-1}(RT)\leq e_j(R)e_k(T).
\]
\item[{\hfill\bf (iv)\hfill}] \textbf{$($Compactness$)$} $T$ is compact if, and only if, \[\lim_{k\rightarrow \infty}e_k(T)=0.\]
\eit
\end{lemma}

\remark{
As for the general theory 
we refer to \cite{EE}, \cite{pie87} and \cite{koe86}. 
Further information on the subject is also covered by the more recent books \cite{ET} and \cite{cs90}. 
}

Some problems about entropy numbers of compact embeddings for function spaces can be transferred to corresponding questions in related sequence spaces. 
Let $n>0$ and $\{M_j\}_{j\in\nat_0}$ be a sequence of natural numbers satisfying
\beq\label{restr-Mj}
M_j\sim 2^{jn}, \qquad j\in\nat_0.
\eeq

Concerning entropy numbers for the respective sequence spaces  $b_{p,q}^{s,\varrho}(M_j)$, which are defined as the sequence spaces $b_{p,q}^{s,\varrho}$ in \eqref{defi-Bq} with the sum over $m\in \zn$  replaced by a sum over $m=1,\ldots, M_j$, the following result was proved in \cite[Prop.~3.4]{sch11a}

\begin{proposition}\label{prop-id-comp}
Let $d>0$, $0<\sigma_1,\sigma_2<\infty$, and $0<q_1,q_2\leq\infty$. Furthermore, let $\varrho_1>\varrho_2\geq 0$, 
\beq\label{sigma-1_2}
0<p_1\leq p_2\leq\infty \qquad \text{and}\qquad \delta=\sigma_1-\sigma_2-n\left(\frac{1}{p_1}-\frac{1}{p_2}\right)>0.
\eeq
Then the identity map
\beq
\id: b^{\sigma_1, \varrho_1}_{p_1,q_1}(M_j)\rightarrow b^{\sigma_2, \varrho_2}_{p_2,q_2}(M_j)
\eeq
is compact, where $M_j$ is restricted by \eqref{restr-Mj}.
\end{proposition}

The next theorem provides a sharp result for entropy numbers of the identity operator related to the sequence spaces $b^{s,\varrho}_{p,q}(M_j)$.

\begin{theorem}\label{entropy-n-seq}
 Let $n>0$, $0<s_1,s_2<\infty$, and $0<q_1,q_2\leq\infty$. Furthermore, let $\varrho_1>\varrho_2\geq 0$, 
\beq
0<p_1\leq p_2\leq\infty \qquad \text{and}\qquad \delta=s_1-s_2-n\left(\frac{1}{p_1}-\frac{1}{p_2}\right)>0.
\eeq
For the entropy numbers $e_k$ of the compact operator
\beq
\id: b^{s_1, \varrho_1}_{p_1,q_1}(M_j)\rightarrow b^{s_2, \varrho_2}_{p_2,q_2}(M_j)
\eeq
we have
\[
e_k(\id)\sim k^{-\frac{\delta}{n}+\frac{1}{p_2}-\frac{1}{p_1}}, \qquad k\in\nat.
\]
\end{theorem}

\remark{
The proof of Theorem \ref{entropy-n-seq} follows from \cite[Th.~9.2]{T-frac}. Using the notation from this book we have
\[
b^{s_i, \varrho_i}_{p_i,q_i}(M_j)=\ell_\infty\left[2^{\varrho_i}\ell_{q_i}\left(2^{j(s_i-\frac{n}{p_i})}\ell_{p_i}^{M_j}\right)\right], \qquad i=1,2.
\]
}

Recall the embedding assertions for Besov spaces {$\mathbf{B}^s_{p,q}(\Omega)$} from Proposition \ref{Bcal-emb}. 
We will give an upper bound for the corresponding entropy numbers of these embeddings. For our purposes it will be sufficient to assume $\Omega=B_R$.


\begin{theorem}\label{th-entropy-num}
Let 
\[
0< s_2<s_1<\infty, \quad 0<p_1,p_2\leq\infty,\quad  0<q_1, q_2\leq\infty,
\] 
and 
\[
\delta_+=s_1-s_2-n\left(\frac{1}{p_1}-\frac{1}{p_2}\right)_+>0.
\]
Then the embedding 
\beq\label{id-compact}
\id: \mathbf{B}^{s_1}_{p_1,q_1}(\Omega)\rightarrow \mathbf{B}^{s_2}_{p_2,q_2}(\Omega)
\eeq
is compact and for the related entropy numbers we compute
\beq\label{entrop-1}
e_k(\id)\lesssim k^{-\frac{s_1-s_2}{n}}, \qquad k\in\nat.
\eeq
\end{theorem}

\begin{proof}
 \underline{Step 1:} Let $p_2\geq p_1$, $\delta_+=\delta$, and let $f\in\mathbf{B}^{s_1}_{p_1,q_1}(\Omega)$, then by \cite[Th.~6.1]{dVS93} there is a (nonlinear) bounded extension operator
\beq \label{final_trace}
g=\ex f\qquad \text{such that}\qquad \Res_{\Omega} g=g\Big|_{\Omega}=f
\eeq
and 
\[
\|g|\mathbf{B}^{s_1}_{p_1,q_1}(\rn)\|\leq c\|f|\mathbf{B}^{s_1}_{p_1,q_1}(\Omega)\|.
\]
We may assume that $g$ is zero outside a fixed neighbourhood $\Lambda$ of $\Omega$. 
Using the subatomic approach for $\mathbf{B}^{s_1}_{p_1,q_1}(\rn)$, cf. Remark \ref{HN07}, we can find an optimal decomposition of $g$, i.e., 
\beq\label{oad}
g(x)=\sum_{\beta\in\nat_0^n}\sum_{j=0}^\infty \sum_{m\in\zn}\lambda^\beta_{j,m}k^\beta_{j,m}(x), \qquad 
\|g|\mathbf{B}^{s_1}_{p_1,q_1}(\rn)\|\sim \|\lambda|b^{s_1,\varrho_1}_{p_1,q_1}\|
\eeq
with $\varrho_1>0$ large. 

Let $M_j$ for fixed $j\in\nat_0$ be the number of cubes $Q_{j,m}$ such that 
\[
rQ_{j,m}\cap\Omega\neq \emptyset.
\]
Since $\Omega\subset \rn$ is bounded we have
\[
M_j\sim 2^{jn}, \qquad j\in\nat_0.
\]
This coincides with \eqref{restr-Mj}. We introduce the (nonlinear) operator $S$,
\[
S: \mathbf{B}^{s_1}_{p_1,q_1}(\rn)\rightarrow b^{s_1,\varrho_1}_{p_1,q_1}(M_j)
\]
by
\[
Sg=\lambda, \qquad \lambda=\left\{\lambda^\beta_{j,m}: \beta\in\nat_0^n, j\in\nat_0, m\in\zn, rQ_{j,m}\cap\Omega\neq\emptyset \right\},
\]
where $g$ is given by \eqref{oad}. Recall that the expansion is not unique but this does not matter. It follows that $S$ is a bounded map since
 \[
 \|S\|=\sup_{g\neq 0}\frac{\|\lambda|b^{s_1,\varrho_1}_{p_1,q_1}(M_j)\|}{\|g|\mathbf{B}^{s_1}_{p_1,q_1}(\rn)\|}\leq c.
 \]
 Next we construct the linear map $T$,
 \[
 T: b^{s_2,\varrho_2}_{p_2,q_2}(M_j)\rightarrow \mathbf{B}^{s_2}_{p_2,q_2}(\rn),
 \]
 given by
 \[
 T\lambda=\sum_{\beta\in\nat_0^n}\sum_{j=0}^\infty \sum_{m=1}^{M_j}\lambda^{\beta}_{j,m}k^\beta_{j,m}(x).
 \]
  It follows that $T$ is a linear (since the subatomic approach provides an expansion of functions via universal building blocks) and bounded map,
 \[
  \|T\|=\sup_{\lambda\neq 0}\frac{\|T\lambda|\mathbf{B}^{s_2}_{p_2,q_2}(\rn)\|}{\|\lambda|b^{s_2,\varrho_2}_{p_2,q_2}(M_j)\|}\leq c.
  \]
  We complement the three bounded maps $\ex$, $S$, $T$ by the identity operator
  \beq\label{comp-id}
  \id: b^{s_1, \varrho_1}_{p_1,q_1}(M_j)\rightarrow b^{s_2, \varrho_2}_{p_2,q_2}(M_j)\qquad \text{with}\qquad \varrho_1>\varrho_2,
  \eeq 
  which is compact by Proposition \ref{prop-id-comp} and the restriction operator 
  \[
  \Res_\Omega: \mathbf{B}^{s_2}_{p_2,q_2}(\rn)\rightarrow \mathbf{B}^{s_2}_{p_2,q_2}(\Omega),
  \]
  which is continuous.  From the constructions it follows that 
 \beq
 \id\left(\mathbf{B}^{s_1}_{p_1,q_1}(\Omega)\rightarrow \mathbf{B}^{s_2}_{p_2,q_2}(\Omega)\right)=\Res_\Omega\circ T\circ\id\circ S\circ \ex.
 \eeq

Hence, taking finally $\Res_\Omega$ we obtain $f$ by \eqref{final_trace}, where we started from.
In particular, due to the fact that we used the subatomic approach, the final outcome is independent of ambiguities in the nonlinear constructions $\ex$ and $S$. The unit ball in $\mathbf{B}^{s_1}_{p_1,q_1}(\Omega)$ is mapped by $S\circ \ex$ into a bounded set in 
\[
b^{s_1,\varrho_1}_{p_1,q_1}(M_j).
\] 
Since the identity operator $\id$ from \eqref{comp-id} is compact, this bounded set is mapped into a pre-compact set in 
\[
b^{s_2, \varrho_2}_{p_2,q_2}(M_j),
\]
which can be covered by $2^k$ balls of radius $ce_k(\id)$ with 
\[
e_k(\id)\leq c k^{-\frac{\delta}{n}+\frac{1}{p_2}-\frac{1}{p_1}}, \qquad k\in\nat.
\]
This follows from Theorem \ref{entropy-n-seq}, where we used $p_2\geq p_1$. Applying the two linear and bounded maps $T$ and $\Res_\Omega$ afterwards does not change this covering assertion -- using Lemma \ref{lemma-entropy}(iii) and ignoring constants for the time being.
 Hence, we arrive at a covering of the unit ball in $\mathbf{B}^{s_1}_{p_1,q_1}(\Omega)$ by $2^k$ balls of radius $ce_k(\id)$ in $\mathbf{B}^{s_2}_{p_2,q_2}(\Omega)$. Inserting 
\[
\delta=s_1-s_2-n\left(\frac{1}{p_1}-\frac{1}{p_2}\right)
\] 
in the exponent we finally obtain the desired estimate 
\[
e_k(\id)\leq c k^{-\frac{s_1-s_2}{n}}, \qquad k\in\nat.
\]

\underline{Step 2:} Let $p_1>p_2$. Since by Proposition \ref{Bcal-emb},
\[
\mathbf{B}^{s_2}_{p_1,q_2}(\Omega)\subset\mathbf{B}^{s_2}_{p_2,q_2}(\Omega), 
\]
we see that
\[
\mathbf{B}^{s_1}_{p_1,q_1}(\Omega)\subset\mathbf{B}^{s_2}_{p_1,q_2}(\Omega)\subset\mathbf{B}^{s_2}_{p_2,q_2}(\Omega),
\]
and therefore \eqref{entrop-1} is a consequence of Step 1 applied to $p_1=p_2$. This completes the proof for the upper bound.\\
\smallskip

\end{proof}

\remark{\label{rem-comp}
By \eqref{B-F-B} and the above definitions we have
\beq\label{B-F-B-2}
\mathbf{B}^s_{p,\min(p,q)}(\Omega)\hookrightarrow \Fd(\Omega)\hookrightarrow \mathbf{B}^s_{p,\max(p,q)}(\Omega).
\eeq
In other words, any assertion about entropy numbers for B-spaces where the parameter $q$ does not play any role applies also to the related F-spaces. \\

Therefore, using Lemma \ref{lemma-entropy}(iv) and Theorem \ref{th-entropy-num} we deduce compactness of the corresponding embeddings related to B- and F-spaces under investigation. 
}

\section{Homogeneity}
\label{S-2}

Our first aim is to prove the following characterization.
\begin{proposition}\label{prop:support}
Let $0<p,q\leq \infty$, $s>0$ and let $R>0$ be a real number. Then
$$
\|f|{\bf B}^s_{p,q}(\R^n)\|\sim \left(\int_0^\infty t^{-sq}\omega_r(f,t)_p^q\frac{dt}{t}\right)^{1/q}
$$
for all $f\in {\bf B}^s_{p,q}(\R^n)$ with $\supp f\subset B_R$.
\end{proposition}
\begin{proof}

We shall need, that ${\bf B}^s_{p,q}(B_R)$ embeds compactly into $L_p(B_R)$. This follows at once from the fact that
${\bf B}^s_{p,q}(B_R)$ is compactly embedded into ${\bf B}^{s-\varepsilon}_{p,q}(B_R)$, cf. Remark \ref{rem-comp},  
and ${\bf B}^{s-\varepsilon}_{p,q}(B_R)\hookrightarrow L_p(B_R)$, which is trivial.

We argue similarly to \cite{CLT07}. We have to prove, that
$$
\|f|L_p(\R^n)\|\lesssim \left(\int_0^\infty t^{-sq}\omega_r(f,t)_p^q\frac{dt}{t}\right)^{1/q}
$$
for every $f\in {\bf B}^s_{p,q}(\R^n)$ with $\supp f\subset B_R$. Let us assume, that this is not true.
Then we find a sequence $(f_j)_{j=1}^\infty\subset {\bf B}^s_{p,q}(\R^n)$, such that
\begin{equation}\label{eq:prop:1}
\|f_j|L_p(\R^n)\|=1\quad\text{and}\quad \left(\int_0^\infty t^{-sq}\omega_r(f_j,t)_p^q\frac{dt}{t}\right)^{1/q}\le \frac{1}{j}, 
\end{equation}
i.e., we obtain that $\|f_j|{\bf B}^s_{p,q}(\R^n)\|$ is bounded. The trivial estimates
\begin{align*}
\|f_j|L_p(\R^n)\|&=\|f_j|L_p(B_R)\|\quad \text{and}\quad \|f_j|{\bf B}^s_{p,q}(B_R)\|\le \|f_j|{\bf B}^s_{p,q}(\R^n)\|
\end{align*}
imply that this is true also for $\|f_j|{\bf B}^s_{p,q}(B_R)\|$. Due to the compactness of 
${\bf B}^s_{p,q}(B_R) \hookrightarrow L_p(B_R)$, we may assume, that $f_j\to f$ in $L_p(B_R)$
with $\|f|L_p(B_R)\|=1$. Using the subadditivity of $\omega(\cdot,t)_p$, we obtain that
$$
\left(\int_0^\infty t^{-sq}\omega_r(f_j-f_{j'},t)_p^q\frac{dt}{t}\right)^{1/q}\le \frac{1}{j}+\frac{1}{j'}.
$$
Together with the estimate $\|f_j-f_{j'}|L_p(\R^n)\|\to 0$, this implies that $(f_j)_{j=1}^\infty$
is a Cauchy sequence in ${\bf B}^s_{p,q}(\R^n)$, i.e. $f_j\to g$ in ${\bf B}^s_{p,q}(\R^n)$. Obviously, $f=g$ follows.

The subadditivity of $\omega(\cdot,t)_p$ used to the sum $(f-f_j)+f_j$ implies finally, that
$$
\left(\int_0^\infty t^{-sq}\omega_r(f,t)_p^q\frac{dt}{t}\right)^{1/q}=0.
$$
As $\omega_r(f,t)$ is a non-decreasing function of $t$, this implies that $\omega_r(f,t)=0$ for all $0<t<\infty$
and finally $\|\Delta_h^r f|L_p(\R^n)\|=0$ for all $h\in\R^n$. By standard arguments, this is satisfied only if
$f$ is a polynomial of order at most $r$. Due to its bounded support, we conclude, that $f=0$, which is a contradiction with
$\|f|L_p(\R^n)\|=1.$

\end{proof}
With the help of this proposition, the proof of homogeneity quickly follows.

\begin{theorem}\label{hom-B}
Let $0<\lambda\le 1$ and $f\in {\bf B}^s_{p,q}(\R^n)$ with $\supp f\subset B_\lambda$. Then
\beq\label{hom-B-eq}
\|f(\lambda \cdot)|{\bf B}^s_{p,q}(\R^n)\|\sim \lambda^{s-n/p}\|f|{\bf B}^s_{p,q}(\R^n)\|
\eeq
with constants of equivalence independent of $\lambda$ and $f$.
\end{theorem}
\begin{proof}
We know from Proposition \ref{prop:support} that
$$
\|f(\lambda\cdot)|{\bf B}^s_{p,q}(\R^n)\|\sim \left(\int_0^\infty t^{-sq}\omega_r(f(\lambda\cdot),t)_p^q\frac{dt}{t}\right)^{1/q},
$$
as $\supp f(\lambda\cdot)\subset B_1$. Using $\Delta_h^r(f(\lambda\cdot))(x)=(\Delta ^r_{\lambda h}f)(\lambda x)$, we get
\begin{align*}
\omega_r(f(\lambda\cdot),t)_p&=\sup_{|h|\le t}\|\Delta_h^r(f(\lambda\cdot))\|_p
=\sup_{|h|\le t}\|(\Delta ^r_{\lambda h}f)(\lambda \cdot)\|_p
=\lambda^{-n/p}\sup_{|h|\le t}\|(\Delta ^r_{\lambda h}f)(\cdot)\|_p\\
&=\lambda^{-n/p}\sup_{|\lambda h|\le \lambda t}\|(\Delta ^r_{\lambda h}f)(\cdot)\|_p
=\lambda^{-n/p}\omega_r(f,\lambda t)_p,
\end{align*}
which finally implies
\begin{align*}
\left(\int_0^\infty t^{-sq}\omega_r(f(\lambda\cdot),t)_p^q\frac{dt}{t}\right)^{1/q}
&=\lambda^{-n/p}\left(\int_0^\infty t^{-sq}\omega_r(f,\lambda t)_p^q\frac{dt}{t}\right)^{1/q}\\
&=\lambda^{s-n/p}\left(\int_0^\infty t^{-sq}\omega_r(f,t)_p^q\frac{dt}{t}\right)^{1/q}
\sim \lambda^{s-n/p}\|f|{\bf B}^s_{p,q}(\R^n)\|.
\end{align*}
\end{proof}

The homogeneity property for Triebel-Lizorkin spaces $\Fd(\rn)$ follows similarly.

\begin{proposition}\label{prop:support2}
Let $0<p<\infty, 0<q\le\infty$, $s>0$ and let $R>0$ be a real number. Then
$$
\|f|{\bf F}^s_{p,q}(\R^n)\|\sim \left\|\left(\int_0^\infty t^{-sq}d^r_{t,p}f(\cdot)^q\frac{dt}{t}\right)^{1/q}|L_p(\R^n)\right\|
$$
for all $f\in {\bf F}^s_{p,q}(\R^n)$ with $\supp f\subset B_R$.
\end{proposition}
\begin{proof}
We have to prove that
$$
\|f|L_p(\R^n)\|\lesssim \left\|\left(\int_0^\infty t^{-sq}d^r_{t,p}f(\cdot)^q\frac{dt}{t}\right)^{1/q}|L_p(\R^n)\right\|
$$
for every $f\in {\bf F}^s_{p,q}(\R^n)$ with $\supp f\subset B_R$. Let us assume again, that this is not true. Then we find a sequence
$(f_j)_{j=1}^\infty\subset {\bf F}^s_{p,q}(\R^n)$ such that
$$
\|f_j|L_p(\R^n)\|=1\quad\text{and}\quad \left\|\left(\int_0^\infty t^{-sq}d^r_{t,p}f_j(\cdot)^q\frac{dt}{t}\right)^{1/q}|L_p(\R^n)\right\|\le\frac 1j,
$$
which in turn implies, that $\|f_j|{\bf F}^s_{p,q}(\R^n)\|$ is bounded. Again, the same is true also for $\|f_j|{\bf F}^s_{p,q}(B_R)\|$.
Due to the compactness of ${\bf F}^s_{p,q}(\R^n)\hookrightarrow L_p(\R^n)$ we may assume, that $f_j\to f$ in $L_p(B_R)$ with $\|f|L_p(B_R)\|=1.$
A straightforward calculation shows again that $(f_j)_{j=1}^\infty$ is a Cauchy sequence in ${\bf F}^s_{p,q}(\R^n)$ and, therefore, $f_j\to f$
also in ${\bf F}^s_{p,q}(\R^n)$. Finally, we obtain
$$
\left\|\left(\int_0^\infty t^{-sq}d^r_{t,p}f(\cdot)^q\frac{dt}{t}\right)^{1/q}|L_p(\R^n)\right\|=0
$$
or, equivalently, 
$$
\int_0^\infty t^{-sq}d^r_{t,p}f(x)^q\frac{dt}{t}=0
$$
for almost every $x\in\R^n$. Hence, $d^r_{t,p}f(x)=0$ for almost all $x\in\R^n$ and almost all $t>0$. By standard arguments
it follows that $f$ must be almost everywhere equal to a polynomial of order smaller then $r$. Together with the bounded support of
$f$, we obtain that $f$ must be equal to zero almost everywhere.
\end{proof}

\begin{theorem}\label{hom-F}
Let $0<\lambda\le 1$ and $f\in {\bf F}^s_{p,q}(\R^n)$ with $\supp f\subset B_\lambda$. Then
\beq\label{hom-F-eq}
\|f(\lambda \cdot)|{\bf F}^s_{p,q}(\R^n)\|\sim \lambda^{s-n/p}\|f|{\bf F}^s_{p,q}(\R^n)\|
\eeq
with constants of equivalence independent of $\lambda$ and $f$.
\end{theorem}
\begin{proof}
We know from Proposition \ref{prop:support2} that
$$
\|f(\lambda\cdot)|{\bf F}^s_{p,q}(\R^n)\|\sim \left\|\left(\int_0^\infty t^{-sq}d^r_{t,p}(f(\lambda\cdot))(\cdot)^q\frac{dt}{t}\right)^{1/q}|L_p(\R^n)\right\|,
$$
as $\supp f(\lambda\cdot)\subset B_1$. Using $\Delta_h^r(f(\lambda\cdot))(x)=(\Delta ^r_{\lambda h}f)(\lambda x)$, we get using the substitution $\tilde h=\lambda h$
\begin{align*}
d^r_{t,p}(f(\lambda\cdot))(x)&=\left(t^{-n}\int_{|h|\le t}|\Delta^r_{h}f(\lambda\cdot)(x)|^pdh\right)^{1/p}
=\left(t^{-n}\int_{|h|\le t}|(\Delta^r_{\lambda h}f)(\lambda x)|^pdh\right)^{1/p}\\
&=\left((\lambda t)^{-n}\int_{|\tilde h|\le \lambda t}|(\Delta^r_{\tilde h}f)(\lambda x)|^pd\tilde h\right)^{1/p}=d^r_{\lambda t,p}(f)(\lambda x),
\end{align*}
which finally implies
\begin{align*}
&\left\|\left(\int_0^\infty t^{-sq}d^r_{t,p}(f(\lambda\cdot))(\cdot)^q\frac{dt}{t}\right)^{1/q}|L_p(\R^n)\right\|
=\left\|\left(\int_0^\infty t^{-sq}d^r_{\lambda t,p}f(\lambda \cdot)^q\frac{dt}{t}\right)^{1/q}|L_p(\R^n)\right\|\\
&\qquad\qquad=\lambda^s\left\|\left(\int_0^\infty t^{-sq}d^r_{t,p}f(\lambda \cdot)^q\frac{dt}{t}\right)^{1/q}|L_p(\R^n)\right\|
=\lambda^{s-n/p}\left\|\left(\int_0^\infty t^{-sq}d^r_{t,p}f(\cdot)^q\frac{dt}{t}\right)^{1/q}|L_p(\R^n)\right\|\\
&\qquad\qquad \sim \lambda^{s-n/p}\|f|{\bf F}^s_{p,q}(\R^n)\|.
\end{align*}
\end{proof}

\section{Pointwise multipliers}
\label{S-PoMu}\label{S-3}

We briefly sketch an application of the above homogeneity results in terms of pointwise multipliers.
A locally integrable function $\varphi$ in $\rn$ is called a \textit{pointwise multiplier} in $\Ad(\rn)$ if
\[
f\mapsto \varphi f
\]
maps the considered space into itself. For further details on the subject we refer to \cite[pp.~201-206]{T-F2} and \cite[Ch.~4]{RuSi96}. Our aim is to generalize Proposition \ref{prop-mult-diffeo}  as a direct consequence of Theorems \ref{hom-B}, \ref{hom-F}. Again let $B_{\lambda}$ be the balls introduced in \eqref{op_ba}. 

\begin{corollary}
Let $s>0$, $0<p,q\leq\infty$ and $0<\lambda\leq 1$. Let $\varphi$ be a function having classical derivatives in $B_{2\lambda}$ up to order $1+[s]$ with
\[
|\uD^{\alpha}\varphi(x)|\leq a\lambda^{-|\gamma|}, \qquad |\gamma|\leq 1+[s], \qquad x\in B_{2\lambda},
\]
for some constant $a>0$. Then $\varphi$ is a pointwise multiplier in $\Bd(B_{\lambda})$, 
\beq\label{pointw-mult}
\|\varphi f|\Bd(B_{\lambda})\|\leq c\|f|\Bd(B_{\lambda})\|, 
\eeq
where $c$ is independent of $f\in \Bd(B_{\lambda})$ and of $\lambda$ (but depends on $a$). 
\end{corollary}

\begin{proof}
By Proposition \ref{prop-mult-diffeo} the function $\varphi(\lambda\cdot)$ is a pointwise multiplier in $\Bd(B_1)$. Then \eqref{pointw-mult} is a consequence of \eqref{hom-B-eq},
\begin{eqnarray*}
\|\varphi f|\Bd(B_{\lambda})\|
\sim \lambda^{-(s-\frac np)}\|\varphi f(\lambda\cdot)|\Bd(B_{1})\|
\lesssim \lambda^{-(s-\frac np)}\| f(\lambda\cdot)|\Bd(B_{1})\|
\sim \|f|\Bd(B_{\lambda})\|.
\end{eqnarray*} 
\end{proof} 
 
\remark{
In terms of Triebel-Lizorkin spaces $\Fd(\rn)$  we obtain corresponding results (assuming $p<\infty$) with the additional restriction on the smoothness parameter $s$ that 
\beq\label{restr-s}
s>n\left(\frac{1}{\min(p,q)}-\frac 1p\right).
\eeq
This follows from the fact that the analogue of Proposition \ref{prop-mult-diffeo} for F-spaces is established using an atomic 
characterization of the spaces $\Fd(\rn)$ which is only true if we impose \eqref{restr-s}, cf. \cite[Prop.~9.14]{T-F3}. 

}
 
%
%


\def\cprime{$'$}

\vfill

{\small
\begin{minipage}[t]{0.45\textwidth}
\noindent
Cornelia Schneider\\
Applied Mathematics III\\
University Erlangen--Nuremberg\\
Cauerstra\ss{}e 11\\
91058 Erlangen\\
Germany\\[1ex]
{\tt schneider@am.uni-erlangen.de}
\end{minipage}\hfill
\begin{minipage}[t]{0.45\textwidth}
\noindent
Jan Vyb\'iral\\
Austrian Academy of Sciences (RICAM)\\
Altenbergerstra\ss{}e 69\\
A-4040 Linz\\
Austria\\[1ex]
{\tt jan.vybiral@oeaw.ac.at}
\end{minipage}\\
}

\end{document}